\documentclass[preprint,12pt]{elsarticle}

\usepackage{amssymb}
\usepackage{amsthm}
\usepackage{graphics}
\usepackage{epsfig}
\usepackage{amssymb}
\usepackage{hyperref}
\usepackage{graphicx}
\usepackage{subfigure}
\usepackage{color}
\usepackage{tikz}
\usetikzlibrary{patterns}
\usepackage[ruled,vlined]{algorithm2e}
\usepackage{hyperref}
\usepackage{a4wide}
\usepackage{amsmath}
\usepackage{amsfonts}
\usepackage{enumerate}
\newtheorem*{theoremaux}{Theorem \theoremauxnum}
\gdef\theoremauxnum{1}

\newtheorem{lemma}{\bf Lemma}[section]
 
\newtheorem{theorem}{\bf Theorem}[section]
\newtheorem{proposition}[lemma]{\bf Proposition}
\newtheorem{corollary}[lemma]{\bf Corollary}
\newtheorem{definition}{\bf Definition}[section]

\journal{~}

\begin{document}

\begin{frontmatter}

%% Title, authors and addresses

%% use the tnoteref command within \title for footnotes;
%% use the tnotetext command for the associated footnote;
%% use the fnref command within \author or \address for footnotes;
%% use the fntext command for the associated footnote;
%% use the corref command within \author for corresponding author footnotes;
%% use the cortext command for the associated footnote;
%% use the ead command for the email address,
%% and the form \ead[url] for the home page:
%%\author[rvt]{C.V. ̃Radhakrishnan\corref{cor1}\fnref{fn1}}
%% \title{Title\tnoteref{label1}}
%% \tnotetext[label1]{}
\author{Angsuman Das}
\ead{angsuman.maths@presiuniv.ac.in}
\ead{angsumandas054@gmail.com}
\address{Department of Mathematics,\\ Presidency University, Kolkata, India} 

%\author{}
\author{S. Morteza Mirafzal$^*$}
\ead{mirafzal.m@lu.ac.ir}
\ead{smortezamirafzal@yahoo.com}
\address{Department of Mathematics,\\
		Lorestan University, Khorramabad, Iran} 
\cortext[cor1]{Corresponding author}

\title{Automorphism group of a family of distance regular graphs which are not distance transitive}

\begin{abstract}
	Let $G_n=\mathbb{Z}_n\times \mathbb{Z}_n$ for $n\geq 4$ and $S=\{(i,0),(0,i),(i,i): 1\leq i \leq n-1\}\subset G_n$. Define $\Gamma(n)$ to be the Cayley graph of $G_n$ with respect to the connecting set $S$. It is known that $\Gamma(n)$ is a strongly regular graph with the parameters $(n^2,3n-3,n,6)$ \cite{19}. Hence $\Gamma(n)$ is a distance regular graph.
	It is known that every distance transitive graph is distance regular, but the converse is not true. In this paper, we study some algebraic properties of the graph $\Gamma(n)$. Then by determining the automorphism group of this family of graphs, we show that the graphs under study are not distance transitive. 
\end{abstract}

\begin{keyword}
	%% keywords here, in the form: keyword \sep keyword
	strongly regular graph \sep distance transitive graph   \sep graph automorphism \sep clique
	\MSC[2008] 05C25 \sep 20B25 \sep 05E18
	%% or \MSC[2008] code \sep code (2000 is the default)
	
\end{keyword}
\end{frontmatter}

%%
%% Start line numbering here if you want
%%
% \linenumbers
\section{Introduction and Preliminaries} 
In this paper, a graph $\Gamma=(V,E)$ is considered as an undirected simple graph where $V=V(\Gamma)$ is the vertex set and $E=E(\Gamma)$ is the edge set. For all the terminology and notation not defined here, we follow \cite{1,2,3,4}.

The group of all permutations of a set $V$ is denoted by $Sym(V)$  or just $Sym(n)$ when $|V| =n $. A $permutation$ $group$ $G$ on
$V$ is a subgroup of $Sym(V).$  In this case, we say that $G$ $acts$ on $V$. If $G$ acts on $V$, we say that $G$ is $transitive$ on $V$ (or $G$ acts $transitively$ on $V$) if given any two elements $u$ and $v$ of $V$, there is an element $ \beta $ of  $G$ such that  $\beta (u)= v
$.  If $\Gamma$ is a graph with vertex-set $V$, then we can view each automorphism of $\Gamma$ as a permutation on $V$  and so $Aut(\Gamma) = G$ is a permutation group on $V$.

A graph $\Gamma$ is called $vertex$ $transitive$ if  $Aut(\Gamma)$  acts transitively on $V(\Gamma)$. We say that $\Gamma$ is $edge$-$transitive$ if the group $Aut(\Gamma)$ acts transitively  on the edge set $E$, namely, for any $\{x, y\} ,   \{v, w\} \in E(\Gamma)$, there is some $\pi$ in $Aut(\Gamma)$,  such that $\pi(\{x, y\}) = \{v, w\}$.  We say that $\Gamma$ is $symmetric$ (or $arc$-$transitive$) if  for all vertices $u, v, x, y$ of $\Gamma$ such that $u$ and $v$ are adjacent, and also, $x$ and $y$ are adjacent, there is an automorphism $\pi$ in $Aut(\Gamma)$ such that $\pi(u)=x$ and $\pi(v)=y$. We say that $\Gamma$ is $distance$ $transitive$ if  for all vertices $u, v, x, y$ of $\Gamma$ such that $d(u, v)=d(x, y)$, where $d(u, v)$ denotes the distance between the vertices $u$ and $v$  in $\Gamma$,  there is an automorphism $\pi$ in $Aut(\Gamma)$ such that  $\pi(u)=x$ and $\pi(v)=y.$  The class of distance transitive graphs contains many  of interesting and important graphs. It is easy to see that the cycle $C_n$, the complete graphs $K_n$ and the complete bipartite graph $K_{n,n}$ are distance transitive. Some other interesting examples of distance transitive graphs are the Petersen graph, the crown graph \cite{1,7,13}, Johnson graphs \cite{2,11,12} and hypercube $Q_n$ \cite{1,2,6,15}. Distance-transitive graphs have been extensively studied by various authors. One may find many information about this family of graphs in  \cite{1,2,4,5}.

Let $\Gamma_i(x)$ denote  the set of vertices of $\Gamma$ at distance $i$ from the vertex $x$. Let $\Gamma=(V,E)$ be a simple connected graph with
diameter $D$. A $distance$ $regular$  graph $\Gamma=(V,E)$, with diameter $D$, is a regular connected graph of valency $k$ with the following property. There are positive integers
$$b_0 = k, b_1, ...,b_{D-1};c_1=1, c_2,...,c_D, $$
such that for each pair $(u, v)$  of vertices satisfying $u \in \Gamma_i(v)$, we have
\begin{enumerate}
	\item the number of vertices in $\Gamma_{i-1}(v)$ adjacent to $u$ is $c_i$, $1\leq i \leq D$.
	\item the number of vertices in $\Gamma_{i+1}(v)$ adjacent to $u$ is $b_i$, $0\leq i \leq D-1$.
\end{enumerate}
The intersection array of $\Gamma$ is $i(\Gamma) = \{k,b_1, ...,b_{D-1};  1,c_2, ...,c_d \}$.

It is easy to show that if $\Gamma$ is a distance transitive graph, then it is distance regular \cite{1,4}. For instance, the hypercube $Q_n, \ n>2$ is a  distance transitive,  and hence it is a distance regular graph with the intersection array 
$ \{n,n-1,n-2,...,1; 1,2,3,...,n  \}$ \cite{1}.

Let $\Gamma=(V,E)$ be a graph.
$\Gamma$ is said to be a $strongly$ $regular$ graph with parameters
$(n,k,\lambda,\mu)$, whenever $|V|=n$, $\Gamma$
is a regular graph of valency $k$, every pair of adjacent vertices of $\Gamma$ have $\lambda$ common neighbor(s), and every pair of non adjacent vertices of $\Gamma$ have $\mu$ common neighbor(s). It is clear that the diameter of every strongly regular graph is 2.  Well known examples of strongly regular graphs include the cycle $C_5$,  the Petersen graph and the complete bipartite graph $K_{n,n}$. Note that these graphs are also distance transitive. It is easy to show that if a graph $\Gamma$ is a distance regular graph of diameter 2 and order $n$, with intersection array $(b_0,b_1;c_1,c_2)$, then $\Gamma$ is a strongly regular graph with parameters
$(n,b_0,b_0-b_1-1,c_2).$ Also,  it is not hard  to check that if $\Gamma$ is a strongly regular graph with parameters $(n,k,\lambda,\mu)$, then $\Gamma$ is a distance regular graph with the intersection array $\{k,\lambda-\mu;1,\mu  \}$. 
There are many
papers that study distance regular  graphs and their applications from various points of  view \cite{2,21}.

Let $G$ be any abstract finite group with identity $1$, and
suppose $S$ is a subset of   $G$, with the
properties: $x\in S \Longrightarrow x^{-1} \in S$, and   $1\notin S$. The $Cayley\  graph$  $\Gamma=Cay (G; S)$ is the (simple)
graph whose vertex-set and edge-set  are defined as follows:
$$V(\Gamma) = G ,  E(\Gamma)=\{\{g,h\}\mid g^{-1}h\in S\}$$
It can be shown that the Cayley   graph   $\Gamma=Cay (G; S)$ is connected if and only if the set $S$ generates the group $G$ \cite{1}.

The group $G$ is called a semidirect product of $ N $ by $Q$,
denoted by $ G=N \rtimes Q $, if $G$ contains subgroups $ N $ and $ Q $ such that:  (i) $N \unlhd G $ ($N$ is a normal subgroup of $G$); (ii) $ NQ = G $; and
(iii) $N \cap Q =1 $. 

Almost all known classic families of strongly regular graphs with known automorphism groups are distance transitive. In this paper, we introduce an infinite family of strongly regular graphs $\Gamma(n)$ which are not distance transitive.

\begin{definition}
Let $G_n=\mathbb{Z}_n\times \mathbb{Z}_n$ for $n\geq 4$ and $S=\{(i,0),(0,i),(i,i): 1\leq i \leq n-1\}\subset G_n$. Define $\Gamma(n)$ to be the Cayley graph of $G_n$ with respect to the connecting set $S$.
\end{definition}

It is known that $\Gamma(n)$ is a strongly regular graph with the parameters $(n^2,3n-3,n,6)$ \cite{19}. In the next section, we determine the automorphism group of $\Gamma(n)$ and in the subsequent section, we study different types of transitivity of $\Gamma(n)$.

\section{Automorphism Group of $\Gamma(n)$}\label{Automorphism-Section}
Although in most situations  it is difficult  to determine the automorphism group of a graph, there are various papers in the literature dealing with automorphism groups, and some of the recent works include \cite{1/2-trans, generalized-andrasfai},\cite{7,9,10,11,12,14,15,16,17,18,22}.

Let $\mathcal{G}_n=\mathsf{Aut}(\Gamma(n))$. Consider the neighborhood $N_0$ of the vertex   $(0,0)$. It consists of three cliques, namely,  $C_1=\{(i,0):1\leq i \leq n-1\}$, $C_2=\{(0,i):1\leq i \leq n-1\}$ and $C_3=\{(i,i):1\leq i \leq n-1\}$, each of size $n-1$. (See Figure \ref{0nbd}.)
\begin{figure}[ht]
	\centering
	\begin{center}
		\includegraphics[scale=0.3]{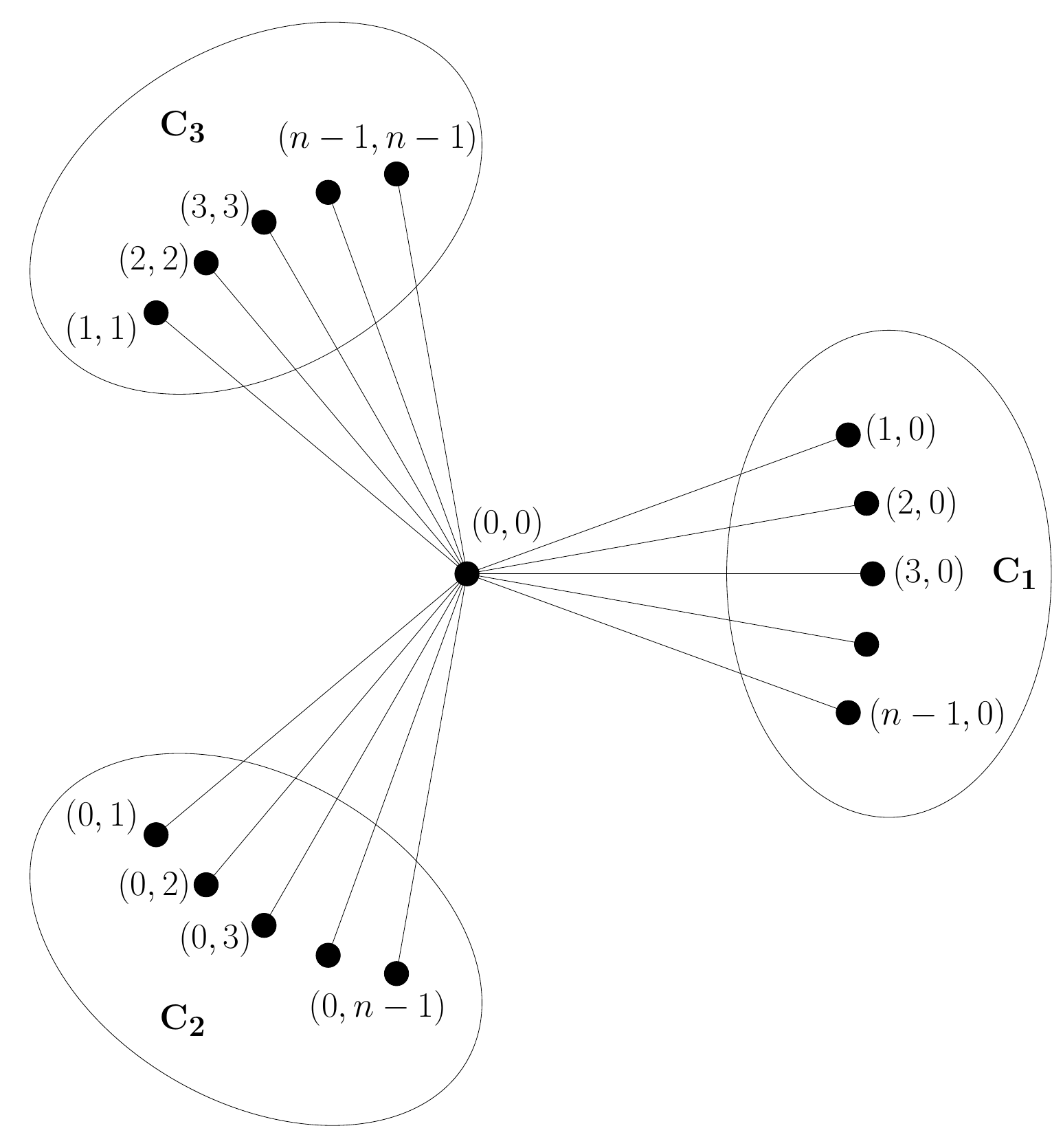}
		\caption{The neighbourhood of $(0,0)$}
		\label{0nbd}
	\end{center}
\end{figure}
Let $\mathcal{G}_0$ be the stabilizer subgroup of $\mathcal{G}_n$ which fixes $(0,0)$. If $f \in \mathcal{G}_0$, then $f(N_0)=N_0$. Let $H(n)$ be the subgraph induced by $N_0$ in the graph $\Gamma(n)$. Thus $g=f|_{N_0}$, the restriction of $f$ to $N_0$,  is an automorphism of the graph $H(n)$.  Since $C_i$ is a maximal clique in $H(n)$, if $v \in C_i$ and $f(v) \in C_j$, then $g(C_i)=C_j$, where   $i,j \in \{1,2,3\}$. Let $C=\{ C_1,C_2,C_3 \}$. We define the function,
\begin{equation}\label{phi-hom}
\varphi : \mathcal{G}_0 \rightarrow Sym(C), \varphi(f)=\varphi_f, \mbox{ where }    \varphi_f(C_i)=f(C_i).
\end{equation}  
It is easy to check that $\varphi$ is a group homomorphism. We now determine the kernel of $\varphi$. Note that if $f \in Ker(\varphi)$, then we have $f(C_i)=C_i$, $1 \leq i \leq 3$.

Let $u$ be a unit (invertible) element in the ring $\mathbb{Z}_n$. It is easy to see that the mapping $\psi_u$ defined on the vertex-set of $\Gamma(n)$ by the $\psi_u(i,j)=(ui,uj)$ is an automorphism of $\Gamma(n)$ such that $\psi_u \in \mathcal{G}_0 $ and is in fact in the kernel of $\varphi$. Let 
$K=\{\psi_u: u \in \mathbb{Z}^*_n\}\cong \mathbb{Z}^*_n$, where $\mathbb{Z}_n^*$ is the group of unit of the ring  $\mathbb{Z}_n$. Thus we have $K \leq Ker(\varphi)$. In the next Lemma \ref{key-lemma}, we show that $K=Ker(\varphi)$. Then we will have $|\frac{\mathcal{G}_0}{Ker(\phi)}| \leq |Sym(C)|$ and hence $|\mathcal{G}_0| \leq 6|K|$. 

\begin{lemma}\label{key-lemma}  Let $f \in \mathcal{G}_0$
	be such that $f \in Ker(\varphi)$, i.e., $f(C_1)=C_1$,  $f(C_2)=C_2$ and $f(C_3)=C_3$.  Then $f \in K=\{\psi_u: u \in \mathbb{Z}^*_n\}\cong \mathbb{Z}^*_n$.
\end{lemma}

\begin{proof}Consider the path $P_i:(i,0)\sim (i,i)\sim (0,i)$ for $1\leq i \leq n-1$. Then we have $f(i,0)=(a_i,0),f(i,i)=(b_i,b_i)$ and $f(0,i)=(0,c_i)$ for some $1\leq a_i,b_i,c_i \leq n-1$. As $f$ maps $P_i$ to another path, by adjacency criterion, we have $a_i=b_i=c_i$, i.e., $$f(i,0)=(c_i,0),f(i,i)=(c_i,c_i),f(0,i)=(0,c_i) \mbox { for }1\leq i \leq n-1.$$
	{\it Claim 1: }$f(i,j)=(c_i,c_j)$ for $1\leq i,j \leq n-1$ and for $i\neq j$, $c_i\neq c_j$.\\
	{\it Proof of Claim 1:} Consider the following neighbourhood of $(i,j)$ (See Figure \ref{ijnbd}). It is mapped to the neighbourhood of some $(x,y)$ as shown in Figure \ref{ijnbd}. From the adjacency relations, we get $(x-c_i,y),(x-c_i,y-c_i),(x,y-c_j),(x-c_j,y-c_j) \in S$.
	
	One can check that the only possible value of $(x,y)$ is $(c_i,c_j)$. Hence Claim 1 holds.
	
	Thus $f(i,j)=(c_i,c_j)$ and for $i\neq j$, $c_i\neq c_j$ and none of $c_i$'s are $0$.

	\begin{figure}[ht]
		\centering
		\begin{center}
			\includegraphics[scale=0.4]{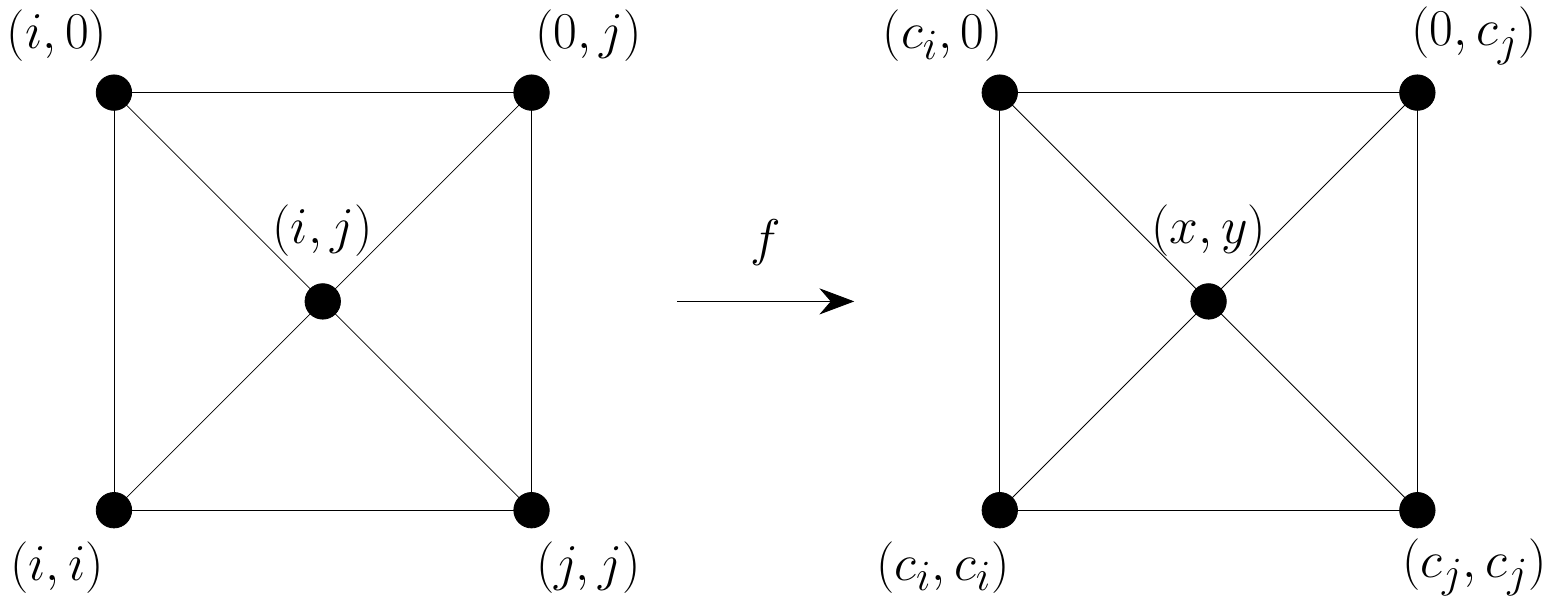}
			\caption{A part of neighbourhood of $(i,j)$ and its image under $f$}
			\label{ijnbd}
		\end{center}
	\end{figure}
	
	{\it Claim 2:} $c_{i+k}-c_i=c_{j+k}-c_j$ for all $i,j,k \in \{1,\ldots,n-1\}$.\\
	{\it Proof of Claim 2:} As $(i,j)\sim (i+k,j+k)$ for $k\neq 0$, we have $(c_i,c_j)\sim (c_{i+k},c_{j+k})$, i.e., $c_{i+k}-c_i=c_{j+k}-c_j$ for all $i,j,k \in \{1,\ldots,n-1\}$.
	
	{\it Claim 3:} If $i+j=n$, then $c_i+c_j\equiv 0~(mod~n)$.\\
	{\it Proof of Claim 3:} As $(i,0)\sim(0,n-i)$, we have $(c_i,0)\sim (0,c_{n-i})$, i.e., $c_i=n-c_{n-i}$, i.e., $c_i+c_{n-i}\equiv 0~(mod~n)$. Replacing $n-i$ by $j$, we get the claim.
	
	From Claim 2, we get 
	\begin{equation}\label{eq1}
	c_2-c_1=c_3-c_2=\cdots=c_{n-1}-c_{n-2}=c~(mod~n) \mbox{ (say)}.
	\end{equation}
	Also from Claim 3, we have 
	\begin{equation}\label{eq2}
	c_1+c_{n-1}=0~(mod~n).
	\end{equation}
	
	{\it Claim 4:} $gcd(c,n)=1$.\\
	{\it Proof of Claim 3:} If $gcd(c,n)>1$, then there exists a positive integer $d<n-1$ such that $dc\equiv 0~(mod~n)$. Thus from the first $d$ equations in Equation \ref{eq1}, we have $c_2-c_1=c_3-c_2=\cdots=c_{d+1}-c_{d}=c~(mod~n)$. Adding these $d$ congruences, we get $c_{d+1}-c_1\equiv dc\equiv 0~(mod~n)$, i.e., $c_1=c_{d+1}$, a contradiction to Claim 1. Hence the claim holds.
	
	{\bf Case 1: $n\geq 5$ is odd.}
	Writing Equations \ref{eq1} and \ref{eq2} in matrix form, we get
	$$\left[\begin{array}{cccccc}
	-1 & 1 & 0 & 0 & \hdots & 0\\
	0 & -1 & 1 & 0 & \hdots & 0\\
	0 & 0 & -1  & 1 & \hdots & 0\\
	\vdots & \vdots & \vdots & \ddots & \vdots & \vdots\\ 
	0 & 0 & 0 & \hdots & -1 & 1\\
	1 & 0 & 0 & \hdots & \hdots & 1\\
	\end{array}\right] 
	\left[\begin{array}{c}
	c_1\\
	c_2\\
	c_3\\
	\vdots\\
	c_{n-2}\\
	c_{n-1}
	\end{array}\right]=
	c\left[\begin{array}{c}
	1\\
	1\\
	1\\
	\vdots\\
	1\\
	0
	\end{array}\right]~(mod~n).$$
	Since the determinant of the square matrix on the left is $\pm 2$, it is non-singular modulo $n$ for all odd $n\geq 5$, the system has a unique solution. Now we note that $$\left[\begin{array}{c}
	c_1\\
	c_2\\
	c_3\\
	\vdots\\
	c_{n-2}\\
	c_{n-1}
	\end{array}\right]=
	c\left[\begin{array}{c}
	1\\
	2\\
	3\\
	\vdots\\
	n-2\\
	n-1
	\end{array}\right]~(mod~n) \mbox{ is a solution to the above system.}$$
	Thus we have $f(i,j)=(ci,cj)$ for some $c \in \mathbb{Z}^*_n$. Thus $f=\psi_c$ for some $c \in \mathbb{Z}^*_n$, i.e., $f \in K$.
	
	{\bf Case 2: $n\geq 4$ is even}.
	Let $n=2m$. We note that as $gcd(c,n)=1$, $c$ is odd. From Claim 4, putting $i=j=m$, we get $c_m=m$. Again from Equations \ref{eq1}, we get $c_{m+1}-c_m=c_m-c_{m+1}=c$, i.e., $c_{m+1}=m+c$ and $c_{m-1}=m-c$. Proceeding similarly and using the fact that $c$ is odd, we get 
	$$\begin{array}{l}
	c_1=m-(m-1)c=m(1-c)+c\equiv c~(mod~n)\\
	c_2=m-(m-2)c=m(1-c)+2c\equiv 2c~(mod~n)\\
	\vdots \\
	c_{m-1}=m-c\equiv (m-1)c~(mod~n)\\
	c_m=m\equiv mc~(mod~n)\\
	c_{m+1}=m+c\equiv (m+1)c~(mod~n)\\
	\vdots \\
	c_{n-1}=m+(m-1)c= m(1+c)-c\equiv(n-1)c~(mod~n)
	\end{array},$$ 
	i.e., $c_i=ci$, i.e., $f(i,j)=(ci,cj)$ for some $c \in \mathbb{Z}^*_n$. Thus $f=\psi_c$ for some $c \in \mathbb{Z}^*_n$, i.e., $f \in K$.
\end{proof}

We continue by noting the following automorphisms of $\Gamma(n)$: 
$$\begin{array}{ll}
\sigma: (i,j) \mapsto (j,i); & \\
\alpha: (i,j) \mapsto (-j,i-j); & \\
\end{array}$$
where all the operations are done modulo $n$. One can easily check that  $\sigma,\alpha \in \mathcal{G}_n=Aut(\Gamma_n)$. In fact $\alpha$ and $\sigma$ are  automorphisms of the group $G_n=\mathbb{Z}_n\times \mathbb{Z}_n$ which stabilize the connection set $S$. Also note that $\alpha(i,0)=(0,i-0)=(0,i)$, $\alpha(0,i)=(-i,0-i)=(-i,-i)$, $\alpha(i,i)=(i,i-i)=(i,0)$. Thus we have $\alpha(C_1)=C_2$, $\alpha(C_2)=C_3$, $\alpha(C_3)=C_1.$
 
Let   $L=\langle \alpha,\sigma: \circ(\sigma)=2;\circ(\alpha)=3; \sigma\alpha\sigma=\alpha^2 \rangle$.  Thus we have $ L \cong Sym(3) \cong Sym(C)$. \\
We now can deduce that every element $g$ of the subgroup $L$ is of the form $g=\alpha^i \sigma^j$, $i \in \{ 0,1,2 \}, \ j \in \{ 0,1 \}$. Clearly $K$ and $L$ are subgroups of $\mathcal{G}_0$ and it is not hard to check that  $K\cap L= \{  1\}$, where $1$ denotes the identity automorphism.
Moreover, one can check that the following relations hold:
$$\psi_u\circ \sigma=\sigma\circ \psi_u, \ \psi_u\circ \alpha=\alpha\circ \psi_u.$$
As elements of $K$ and $L$ commute with each other, $KL\cong K \times L$ is a subgroup of $\mathcal{G}_0$ of order $6|K|=6\varphi(n)$ ($\varphi(n)=|\mathbb{Z}_n^*|$). We now have the following result

\begin{corollary}\label{cor-stabilizer}
	$\mathcal{G}_0=KL\cong  K \times L$.
\end{corollary}
\begin{proof}We know that $|\frac{\mathcal{G}_0}{Ker(\varphi)}| \leq |Sym(C)|$ and hence $|\mathcal{G}_0| \leq 6|Ker(\varphi)|$, where $\varphi$ is the group homomorphism defined in Equation \ref{phi-hom}. From Lemma \ref{key-lemma}, it follows that $Ker(\varphi)=K$ and hence we have $|\mathcal{G}_0| \leq 6|K|$. Moreover, we just deduced that $KL\cong K \times L$ is a subgroup of $\mathcal{G}_0$ of order $6|K|$. Thus we have $\mathcal{G}_0=KL\cong  K \times L$.
\end{proof}

From Corollary \ref{cor-stabilizer}, follows the following important result.

\begin{theorem}\label{aut-theorem}
	For $n\geq 4$, $Aut(\Gamma(n))\cong G_n \rtimes (K \times L)\cong  (\mathbb{Z}_n\times \mathbb{Z}_n )\rtimes (\mathbb{Z}_n^* \times Sym(3)).$
\end{theorem}

\begin{proof} As $\Gamma(n)$ is a Cayley graph, it is vertex transitive. Thus by Orbit-Stabilizer Theorem, we have $|V(\Gamma(n))|=|\frac{\mathcal{G}_n}{\mathcal{G}_0}|$, where $\mathcal{G}_n=Aut(\Gamma(n))$. Thus we have $|\mathcal{G}_n|=|G_n||\mathcal{G}_0|$. Hence by Corollary \ref{cor-stabilizer}, we have  $|\mathcal{G}_n|=|G_n||KL|$.  Note that the left regular representation of $G_n$ denoted by $L(G_n)=\{l_g \ | \ g \in G_n   \}$ is a subgroup of the automorphism group of the graph $\Gamma_n$ and  $L(G_n)  \cong G_n$ $(  l_g : G_n \rightarrow G_n, \  l_g(v)=g+v, \  v\in G_n ).$ Every element of  the subgroup $KL$ is an automorphism of the abelian group $G_n$ which normalize the subgroup $L(G_n)$ of $\mathcal{G}_n.$ In fact if $f\in KL$, then we have $(f^{-1}l_gf)(v)=f^{-1}(g+f(v))=f^{-1}(g)+v=l_{f^{-1}(g)}(v)$, for each $v \in G_n$. Thus $f^{-1}l_gf=l_{f^{-1}(g)} \in L(G_n)$, for every $g \in G_n.$ It is easy to see that $ \L(G_n) \cap KL=\{1 \}$. Hence we have,
	$$\langle L(G_n), KL \rangle \cong L(G_n) \rtimes KL$$
	is a subgroup of  $\mathcal{G}_n$ of order $|G_n||KL|$. Now we conclude that,
	$$Aut(\Gamma(n))=\mathcal{G}_n= L(G_n) \rtimes KL \cong G_n \rtimes (\mathbb{Z}_n^* \times Sym(3))\cong (\mathbb{Z}_n\times \mathbb{Z}_n )\rtimes (\mathbb{Z}_n^* \times Sym(3)) .$$
	
\end{proof}

\section{Transitivity of $\mathcal{G}_n$ on $\Gamma(n)$}
Since $\Gamma(n)$ is a Cayley graph on $\mathbb{Z}_n \times \mathbb{Z}_n$, it is vertex-transitive. In this section, we study the edge-transitivity, arc-transitivity, distance-transitivity of $\Gamma(n)$.

We recall that as $\Gamma(n)$ is a Cayley graph on $G_n$, $\{T_{a,b}:a,b \in \mathbb{Z}_n\}\cong \mathbb{Z}_n \times \mathbb{Z}_n$, i.e., the left regular representation of $G_n$ is a subgroup of $\mathcal{G}_n$, where $T_{a,b}: \mathbb{Z}_n \times \mathbb{Z}_n \rightarrow \mathbb{Z}_n \times \mathbb{Z}_n$ is given by $T_{a,b}(x,y)=(x+a,y+b)$ for $x,y \in \mathbb{Z}_n$.

\begin{theorem}
	If $n$ is composite, then $\Gamma(n)$ is not edge-transitive. 
\end{theorem}
\begin{proof}  
	Let $n$ be composite and $1<m<n$ be a factor of $n$. Consider the edges $\overrightarrow{e_1}=(0,0)\sim (m,0)$ and $\overrightarrow{e_2}=(0,0)\sim (1,1)$. We show that there does not exist any automorphism $f$ in $\mathcal{G}_n$ such that $f(\overrightarrow{e_1})=\overrightarrow{e_2}$ or $f(\overrightarrow{e_1})=\overleftarrow{e_2}$. 
	
	If possible, let $f(\overrightarrow{e_1})=\overrightarrow{e_2}$. Then $f \in \mathcal{G}_0$ and hence $f=\psi_c\circ \alpha^i\circ  \sigma^j$, for some $c \in \mathbb{Z}^*_n$, $i \in \{0,1,2\}$ and $j \in \{0,1\}$. As $\sigma(m,0)=\alpha(m,0)=(0,m)$ and $\alpha\sigma(m,0)=\alpha^2(m,0)=(-m,-m)$ and $\alpha^2\sigma(m,0)=(m,0)$, the only possibility of $f$ to get $f(m,0)=(1,1)$ is $f=\psi_{m^{-1}}\alpha\sigma$ or $f=\psi_{m^{-1}}\alpha^2$. However, as $m^{-1}$ does not exist in $\mathbb{Z}^*_n$, there does not exist any automorphism $f$ in $\mathcal{G}_n$ such that $f(\overrightarrow{e_1})=\overrightarrow{e_2}$.
	
	If possible, let $f(\overrightarrow{e_1})=\overleftarrow{e_2}$, i.e., $f(0,0)=(1,1)$ and $f(m,0)=(0,0)$. Let $f=T_{a,b}\psi_c\alpha^i\sigma^j$, for some $a,b \in \mathbb{Z}_n$, $c \in \mathbb{Z}^*_n$, $i \in \{0,1,2\}$ and $j \in \{0,1\}$. As $\psi_c\alpha^i\sigma^j(0,0)=(0,0)$, we have $a=b=1$. Thus $(0,0)=f(m,0)=T_{1,1}\psi_c\alpha^i\sigma^j(m,0)$, i.e., $\psi_c\alpha^i\sigma^j(m,0)=(-1,-1)$, i.e., $\alpha^i\sigma^j(m,0)=(-c^{-1},-c^{-1})$. As $\sigma(m,0)=\alpha(m,0)=(0,m)$ and $\alpha\sigma(m,0)=\alpha^2(m,0)=(-m,-m)$ and $\alpha^2\sigma(m,0)=(-m,0)$, the only possibilities are $(i,j)=(1,1)$ or $(2,0)$. However, as $m^{-1}$ does not exist in $\mathbb{Z}^*_n$, there does not exist any automorphism $f$ in $\mathcal{G}_n$ such that $f(\overrightarrow{e_1})=\overleftarrow{e_2}$. 
	
	Thus $\Gamma(n)$ is not edge-transitive. 
\end{proof}

\begin{theorem}
	If $n$ is prime, then $\Gamma(n)$ is arc-transitive.
\end{theorem}

\begin{proof}
	Since $\Gamma(n)$ is vertex-transitive, for arc-transitivity, it is enough to show that for any two edges $\overrightarrow{e_1}$ and $\overrightarrow{e_2}$ incident to $(0,0)$, there exist automorphisms $f_1,f_2\in \mathcal{G}_n$ such that $f_1(\overrightarrow{e_1})=\overrightarrow{e_2}$ and $f_2(\overrightarrow{e_1})=\overleftarrow{e_2}$.
	
	There are three types of edges incident to $(0,0)$: 
	\begin{itemize}
		\item Type-I: $(0,0)\sim(x,0)$ for some $1\leq x\leq n-1$.
		\item Type-II: $(0,0)\sim(0,x)$ for some $1\leq x\leq n-1$.
		\item Type-III: $(0,0)\sim(x,x)$ for some $1\leq x\leq n-1$.
	\end{itemize}
	
	If both $\overrightarrow{e_1}$ and $\overrightarrow{e_2}$ are of Type-I, with $\overrightarrow{e_1}=(0,0)\sim(x,0)$ and $\overrightarrow{e_2}=(0,0)\sim(y,0)$, then we have $\psi_{yx^{-1}}(\overrightarrow{e_1})=\overrightarrow{e_2}$ and $T_{y,0}\psi_{-xy^{-1}}(\overrightarrow{e_1})=\overleftarrow{e_2}$.
	
	If both $\overrightarrow{e_1}$ and $\overrightarrow{e_2}$ are of Type-II, with $\overrightarrow{e_1}=(0,0)\sim(0,x)$ and $\overrightarrow{e_2}=(0,0)\sim(0,y)$, then we have $\psi_{yx^{-1}}(\overrightarrow{e_1})=\overrightarrow{e_2}$ and $T_{0,y}\psi_{-yx^{-1}}(\overrightarrow{e_1})=\overleftarrow{e_2}$.
	
	If both $\overrightarrow{e_1}$ and $\overrightarrow{e_2}$ are of Type-II, with $\overrightarrow{e_1}=(0,0)\sim(x,x)$ and $\overrightarrow{e_2}=(0,0)\sim(y,y)$, then we have $\psi_{yx^{-1}}(\overrightarrow{e_1})=\overrightarrow{e_2}$ and $T_{y,y}\psi_{-yx^{-1}}(\overrightarrow{e_1})=\overleftarrow{e_2}$.
	
	If $\overrightarrow{e_1}$ is of Type-I and $\overrightarrow{e_2}$ is of Type-II, with $\overrightarrow{e_1}=(0,0)\sim(x,0)$ and $\overrightarrow{e_2}=(0,0)\sim(0,y)$, then we have $\psi_{yx^{-1}}\sigma(\overrightarrow{e_1})=\overrightarrow{e_2}$ and $T_{0,y}\psi_{-yx^{-1}}\sigma (\overrightarrow{e_1})=\overleftarrow{e_2}$.
	
	If $\overrightarrow{e_1}$ is of Type-I and $\overrightarrow{e_2}$ is of Type-III, with $\overrightarrow{e_1}=(0,0)\sim(x,0)$ and $\overrightarrow{e_2}=(0,0)\sim(y,y)$, then we have $\psi_{-yx^{-1}}\alpha\sigma(\overrightarrow{e_1})=\overrightarrow{e_2}$ and $T_{y,y}\psi_{yx^{-1}}\alpha\sigma (\overrightarrow{e_1})=\overleftarrow{e_2}$.
	
	If $\overrightarrow{e_1}$ is of Type-II and $\overrightarrow{e_2}$ is of Type-III, with $\overrightarrow{e_1}=(0,0)\sim(0,x)$ and $\overrightarrow{e_2}=(0,0)\sim(y,y)$, then we have $\psi_{-yx^{-1}}\alpha(\overrightarrow{e_1})=\overrightarrow{e_2}$ and $T_{y,y}\psi_{yx^{-1}}\alpha (\overrightarrow{e_1})=\overleftarrow{e_2}$.
	
	As $n$ is prime, $x^{-1}$ exists in modulo $n$ and hence $\Gamma(n)$ is arc-transitive.
\end{proof}

\begin{theorem}
	If $n\neq 5$, $\Gamma(n)$ is not distance-transitive. 
\end{theorem}

\begin{proof}
	If $n$ is composite, as $\Gamma(n)$ is not edge-transitive, it is not distance transitive. So, we assume $n$ to be prime and $n\geq 7$. Consider the two paths $P_1:(0,0)\sim (0,1)\sim (2,3)$ and $P_2:(0,0)\sim(2,2)\sim(4,2)$. Thus both $(2,3)$ and $(4,2)$ are at distance two from $(0,0)$. 
	
	If $\Gamma(n)$ is distance-transitive, then there exists $f \in \mathcal{G}_n$ such that $f(0,0)=(0,0)$ and $f(2,3)=(4,2)$. As $(0,0)$ is fixed under $f$, $f$ must be of the form $\psi_c\alpha^i\sigma^j$ for some $c \in \mathbb{Z}^*_n$, $i \in \{0,1,2\}$ and $j \in \{0,1\}$.
	
	If $i=j=0$, then we must have $2c\equiv 4~(mod ~n)$ and $3c\equiv 2~(mod~n)$ for some $c \in \mathbb{Z}^*_n$. However existence of such a $c$ implies $n|4$, a contradiction. Thus $(i,j)\neq (0,0)$.
	
	If $(i,j)=(0,1)$, then we have $\psi_c\sigma(2,3)=(4,2)$, i.e., $(3c,2c)=(4,2)$, which has no solution for $c$. Thus $(i,j)\neq (0,1)$.
	
	If $(i,j)=(1,0)$, then we have $\psi_c\alpha(2,3)=(4,2)$, i.e., $(-3c,-c)=(4,2)$, which has no solution for $c$. Thus $(i,j)\neq (1,0)$.
	
	If $(i,j)=(2,0)$, then we have $\psi_c\alpha^2(2,3)=(4,2)$, i.e., $(c,-2c)=(4,2)$. This can happen only if $c=4$ and $n=5$. Thus $(i,j)\neq (2,0)$ if $n\neq 5$.
	
	If $(i,j)=(1,1)$, then we have $\psi_c\alpha\sigma(2,3)=(4,2)$, i.e., $(-2c,c)=(4,2)$, which has no solution for $c$. Thus $(i,j)\neq (1,1)$.
	
	If $(i,j)=(2,1)$, then we have $\psi_c\alpha^2\sigma(2,3)=(4,2)$, i.e., $(-c,-3c)=(4,2)$. This can happen only if $c=-4$ and $n=5$. Thus $(i,j)\neq (2,1)$ if $n\neq 5$.
	
	Hence the theorem holds.
\end{proof}

\begin{proposition}
	$\Gamma(5)$ is distance transitive, but not $2$-arc-transitive.
\end{proposition}

\begin{proof}
	Though it can be checked easily in SageMath  computations [20], for the sake of completeness, we provide a sketch of the proof: Since $\Gamma(5)$ is vertex-transitive and of diameter $2$, for distance-transitivity, it is enough to show that for any two vertices $(x,y)$ and $(u,v)$ which are not adjacent to $(0,0)$, there exists an automorphism $f \in \mathcal{G}_5$ such that $f(0,0)=(0,0)$ and $f(x,y)=(u,v)$.
	
	As $n=5$ and $(x,y)$ and $(u,v)$ are not adjacent to $(0,0)$, both $(x,y)$ and $(u,v)$ must belong to the set 
	$$N= \{(1,2),(1,3),(1,4),(2,1),(2,3),(2,4),(3,1),(3,2),(3,4),(4,1),(4,2),(4,3)\}.$$
	
	As $K\times L$ acts transitively on this set $N$, we can always find such an $f$. Hence $\Gamma(5)$ is distance transitive.
	
	To show that $\Gamma(5)$ is not $2$-arc-transitive, consider the two $2$-arcs $P_1:(0,0)\sim (0,1)\sim (2,3)$ and $P_2:(0,0)\sim(2,2)\sim(4,2)$ respectively. Following the arguments as in previous theorem, it can be shown that there does not exist any automorphism which maps one of the $2$-arcs to the other. 
\end{proof}

\section*{Acknowledgement}
The first author acknowledge the funding of DST-SERB-MATRICS Sanction no.\\ $MTR/2022/000020$, Govt. of India.

\end{document}